\documentclass[11 pt]{amsart}
\usepackage{fullpage, amsthm, amsmath, amsfonts, amssymb, url, mathrsfs, stmaryrd}
\usepackage[colorlinks]{hyperref}

\author{Casey Rodriguez}
\address{Department of Mathematics\\
  University of Chicago\\
  5734 S. University Avenue \\
  Chicago, IL, 60637}
\email{c-rod216@math.uchicago.edu}

\numberwithin{equation}{section}

\newcommand{\R}{\mathbb R}

\newcommand{\ra}{\rangle}
\newcommand{\la}{\langle}

\newcommand{\diver}{\mbox{div}_g}

\newcommand{\nmlder}{\partial_{\nu}}
\newcommand{\grad}{\nabla_g}
\newcommand{\laplace}{\Delta_g}

\newcommand{\vol}{\hspace{2 pt} dV_g}
\newcommand{\tvol}{\hspace{2 pt} dV_{g}}

\newtheorem{lem}{Lemma}[section]
\newtheorem{thm}[lem]{Theorem}
\newtheorem{ppn}[lem]{Proposition}
\newtheorem{defn}[lem]{Definition}
\newtheorem{cor}[lem]{Corollary}

\title{A Partial Data Result for Less Regular Conductivities \\ in Admissible Geometries}
\date{\today}

\begin{document}

\begin{abstract}
We consider the Calder\'on problem with partial data in certain admissible geometries, that is, on compact Riemannian manifolds
with boundary which are conformally embedded in a product of the Euclidean line and a simple manifold. We show that
measuring the Dirichlet--to--Neumann map on roughly half of the boundary determines a conductivity that has essentially
3/2 derivatives. As a corollary, we strengthen a partial data result due to Kenig, Sj\"ostrand, and Uhlmann.
\end{abstract}

\maketitle

\section{Introduction}
In 1980 A. P. Calder\'on published a short paper \cite{cald} where he posed the following question: is it possible to determine the electrical conductivity of a medium by making voltage
and current measurements on the boundary?  This pioneering contribution motivated many developments in inverse problems, in particular the construction
of \lq complex geometrical optics' solutions of partial differential equations to solve several inverse problems.  The precise mathematical
formulation of the problem is the following.

Let $\Omega \subseteq \R^n$, $n \geq 2$, be a bounded domain with smooth boundary.  The electrical conductivity of $\Omega$ is modeled by a bounded
positive function $\gamma$.  In the absence of sources or sinks, given a boundary potential $f \in H^{1/2}(\partial \Omega)$ the induced potential $u \in H^1(\Omega)$
solves the Dirichlet problem
\begin{align*}
\mbox{div}(\gamma \nabla u) &= 0 \quad \mbox{in } \Omega \\
u &= f \quad \mbox{on } \partial \Omega.
\end{align*}
The Dirichlet--to--Neumann map, or voltage--to--current map, is given by
$$
\Lambda_\gamma(f) = \left ( \gamma \partial_\nu u \right ) \Big |_{\partial \Omega}
$$
where $\nu$ denotes the unit outer normal to $\partial \Omega$.  The Calder\'on problem is to determine $\gamma$ from measurements
of $\Lambda_\gamma$.

Substantial progress has been made on the Calder\'on problem since Calder\'on's paper.  See \cite{uhl2} for an exposition of many of the
main results. Uhlmann and Sylvester proved in \cite{uhl1} that for $n \geq 3$, knowledge of $\Lambda_\gamma$
on the whole boundary uniquely determines $\gamma \in C^2(\overline{\Omega})$.  Since then, there has been considerable work on
reducing the assumption that $\gamma \in C^2(\overline{\Omega})$.  In \cite{hab} Haberman and Tataru improved this assumption to
$\gamma \in C^1(\overline{\Omega})$.  Recently in \cite{hab2} Haberman improved this assumption to $\gamma$ with
unbounded gradient in low dimensions, and in \cite{caro} Caro and Rogers proved uniqueness for Lipschitz $\gamma$ in higher dimensions.
For $n = 2$, Astala and P\"aiv\"arinta proved in \cite{ast} that knowledge of $\Lambda_\gamma$
on the whole boundary uniquely determines $\gamma \in L^\infty(\Omega)$.

In the case that $\Lambda_\gamma$ is measured only on part of the boundary, it was first shown by Bukhgeim and Uhlmann in \cite{buk} that for $n \geq 3$,
knowledge of $\Lambda_\gamma$ on roughly half of the boundary determines $\gamma \in C^2(\overline{\Omega})$ uniquely.  The assumptions made in \cite{buk} on the
structure of the subset where $\Lambda_\gamma$ is measured were greatly improved by Kenig et al. in \cite{ksu}, but their assumption on the regularity of the
conductivity was the same. The regularity assumption made on the conductivity in \cite{buk} was improved by Knudsen in \cite{knu} to $\gamma \in W^{3/2+,2n}(\Omega)$.
In \cite{zhang} Zhang used ideas from \cite{hab} to reduce this further to $\gamma \in C^1 (\overline{\Omega})$.
To the author's knowledge, there has been no improvement made to the regularity assumption in \cite{ksu}. By
considering a more general setting for the Calder\'on problem, we will improve the regularity assumption in the result from \cite{ksu}.

In this paper, we replace $\Omega$ with a compact $n$--dimensional Riemannian manifold $(M,g)$.  In particular, let $\gamma$ be a bounded positive function
on $M$. Let $\nu$ denote the unit outer normal to
$\partial M$ with respect to $g$, $\grad$ denote the gradient operator on $M$, $\diver$ denote the divergence operator on $M$, and $|g| = \det g$.
In local coordinates $(x^i)$ where $g = (g_{ij})$,
\begin{align*}
\left ( \grad u \right )^i &= g^{ij} \partial_j u, \\
\diver X &= \frac{1}{|g|^{1/2}} \partial_i \left ( |g|^{1/2} X^i \right ),
\end{align*}
for smooth functions $u$ and vector fields $X = X^i \partial_i$.

  Given $f \in H^{1/2}(\Omega)$, define
$$
\Lambda_{g,\gamma}(f) = \left ( \gamma \partial_\nu u \right ) \Big |_{\partial M}, \
$$
where $u \in H^1(M)$ is the unique solution to the Dirichlet problem
\begin{align}
\diver(\gamma \nabla_g u) &= 0 \quad \mbox{in } M \label{condeq1}\\
u &= f \quad \mbox{on } \partial M. \notag
\end{align}
The inverse problem is to determine $\gamma$ from the knowledge of $\Lambda_{g,\gamma}$.
In order to state our main result, we need the following notion of an \emph{admissible} Riemannian manifold which was introduced
in \cite{dfksu}.

\begin{defn}
A compact Riemannian manifold $(M,g)$ with dimension $n \geq 3$ and with boundary $\partial M$, is called admissible if $M \subseteq \R \times M^{int}_0$
for some $(n-1)$--dimensional simple manifold $(M_0,g_0)$, and if $g = c(e \oplus g_0)$ where $e$ is the Euclidean metric on $\R$ and $c$
is a smooth positive function on $M$.
\end{defn}
Here a compact manifold $(M_0,g_0)$ with boundary is simple if for any $\omega \in M_0$, the exponential map $\exp_\omega$ with its maximal domain of
definition is a diffeomorphism onto $M_0$ and if $\partial M_0$ is strictly convex (that is, the second fundamental form of $\partial M_0
\hookrightarrow M_0$ is positive definite).  Examples of admissible manifolds include compact submanifolds of Euclidean space, the sphere minus a point,
and hyperbolic space.

The following theorem is the main result of this paper.
\begin{thm}\label{thm1}
Let $(M,g)$ be an admissible manifold.  Given $\epsilon > 0$,
define
\begin{align*}
\partial M_{-,\epsilon} = \{ x \in \partial M : \nmlder \varphi(x) < \epsilon \}, \\
\partial M_{+,\epsilon} = \{ x \in \partial M : \nmlder \varphi(x) \geq \epsilon \}
\end{align*}
where $\varphi(x) = x_1$.  Let $\gamma_1,\gamma_2 \in W^{3/2+\eta,2n}(M)$ for some $\eta > 0$. Suppose further that
$\gamma_1|_{\partial M} = \gamma_2|_{\partial M}$ and $\nmlder \gamma_1|_{\partial M} = \nmlder \gamma_2|_{\partial M}$.  Then if for some $\epsilon > 0$
\begin{align*}
\Lambda_{\gamma_1} (f)|_{\partial M_{-,\epsilon}} = \Lambda_{\gamma_2} (f)|_{\partial M_{-,\epsilon}} \mbox{ for any } f \in H^{1/2}(\partial M),
\end{align*}
we have that $\gamma_1 = \gamma_2$.
\end{thm}

This theorem can be seen as a generalization of \cite{knu}.  As a corollary, we obtain a strengthening of Theorem 1.2
from \cite{ksu}.

\begin{thm}\label{thm2}
Let $\Omega \subset \R^n$ be a bounded domain with smooth boundary, and assume $x_0$ is not in the convex hull of $\overline{\Omega}$.
Given $\epsilon > 0$, define
\begin{align*}
\partial \Omega_{-,\epsilon} &= \{ x \in \partial \Omega : \partial_\nu \varphi(x) < \epsilon \}, \\
\partial \Omega_{+,\epsilon} &= \{ x \in \partial \Omega : \partial_\nu \varphi(x) \geq  \epsilon \}.
\end{align*}
where $\varphi(x) = \log|x-x_0|$. Let $\gamma_1,\gamma_2 \in W^{3/2+\eta,2n}(\Omega)$ for some $\eta > 0$. Suppose further that
$\gamma_1|_{\partial \Omega} = \gamma_2|_{\partial \Omega}$ and $\nmlder \gamma_1|_{\partial \Omega} = \nmlder \gamma_2|_{\partial \Omega}$.
Then if for some $\epsilon > 0$
\begin{align*}
\Lambda_{\gamma_1} (f)|_{\partial \Omega_{-,\epsilon}} = \Lambda_{\gamma_2} (f)|_{\partial \Omega_{-,\epsilon}} \mbox{ for any } f \in H^{1/2}(\partial \Omega),
\end{align*}
we have that $\gamma_1 = \gamma_2$.
\end{thm}

In proving Theorem \ref{thm1}, we combine ideas from \cite{knu}, \cite{dfks}, \cite{dfksu}, \cite{ksa}, and \cite{ksau}.  An outline of the proof is as follows.
Let $\gamma \in C^1(M)$ and let $u$ be a solution to \eqref{condeq1}.  After simplifying \eqref{condeq1}, we see that $u$
solves
\begin{align}
(-\Delta_g + \la A , \grad \ra_g) u = 0 \quad \mbox{in } M \label{condeq2}
\end{align}
where
\begin{align}
A = -\grad \log \gamma \label{Adef}.
\end{align}

We first construct a suitable family of complex geometrical optics (CGO) solutions to \eqref{condeq2} using ideas from \cite{knu} and \cite{ksau}.
Next, we use an Alessandrini type integral identity to relate information on the boundary to the attenuated geodesic ray transform of the conductivities.
A uniqueness result for the attenuated ray transform on simple manifolds allows us to conclude that the two conductivities agree on $M$.

The outline of this paper is as follows.  In Section 2, we obtain Carleman estimates for lower order perturbations of the Laplacian $\laplace$.
In Section 3, we construct CGO solutions to \eqref{condeq2}.  Finally, in Section 4 we prove Theorem \ref{thm1} and and Theorem \ref{thm2}.

\textbf{Acknowledgments:}  This work was completed during the author's doctoral studies.  The author would like
to thank his advisor, Carlos Kenig, for suggesting the problem and for his invaluable patience and guidance.

\section{Carleman Estimate}
Throughout the remainder of the paper, $(M,g)$ is an admissible manifold with $g = e \oplus g_0$ and $\varphi(x) = x_1$.  In particular, it suffices to prove Theorem \ref{thm1} in the case
$c = 1$.  This reduction follows easily from the relations
\begin{align}
\mbox{div}_{cg} \left (\gamma \nabla_{cg} u \right ) = c^{-\frac{n}{2}} \mbox{div}_g \left ( \tilde \gamma \grad u \right ), \notag
\end{align}
where
$$
\tilde \gamma = c^{\frac{n-2}{2}} \gamma,
$$
and
$$
\left (\partial_\nu u \right )_{cg} = c^{-1/2} \left (\partial_\nu u \right )_{g},
$$
where the subscripts indicate that the normal derivatives are taken with respect to $cg$ and $g$ respectively.
We also denote the inner product and norm (on tangent spaces) given by $g$ by $\la \cdot, \cdot \ra_g$ and $|\cdot|_g$ respectively.  Finally, the letter $C$ (with possible superscripts) denotes a generic positive constant which may change from line to line.
\begin{ppn}
There exist positive constants $\tau_0, C, C', C''$ such that for all $|\tau| \geq \tau_0$ and all
$v \in H^2(M)$ we have the
estimate
\begin{align}
C&\left ( \tau^2 \int_M |v|^2 \tvol + \int_M |\grad v |_{g}^2 \vol \right ) - C' \tau^2 \int_{\partial M} |v|^2 \hspace{2 pt} dS_{g} - C'' \int_{\partial M} \bar v \nmlder v \hspace{2 pt} dS_{g} \notag\\
&+ \int_{\partial M} \left ( 4 \tau \Re (\nmlder v \partial_1 \bar v) - 2 \tau (\nmlder \varphi) | \grad v |_{g}^2 + 2 \tau^3 (\nmlder \varphi) |v|^2 \right ) \hspace{2 pt} dS_{g} \notag\\
&\leq \int_M |e^{-\tau \varphi } (-\laplace) (e^{\tau \varphi} v) |^2 \tvol. \label{carlest1}
\end{align}
\end{ppn}

\begin{proof}
The proof is similar to the proof of Proposition 2.1 in \cite{knu}.
For $v \in H^2(M)$ we write
\begin{align*}
e^{-\tau x_1} (-\laplace) (e^{\tau x_1}v) = (-\laplace - \tau^2)v - 2 \tau \partial_1 v.
\end{align*}
Hence
\begin{align}
\int_M |e^{-\tau \varphi } (-\laplace) (e^{\tau \varphi} v) |^2 \tvol =& \int_M |(-\laplace - \tau^2)v|^2 \tvol + \int_M |2 \tau \partial v|^2 \tvol \notag \\
&+ 4 \tau \Re \int_M (\laplace + \tau^2) v )\overline{\partial_1 v} \tvol. \label{estone}
\end{align}
The Poincar\'e inequality with boundary term implies that
\begin{align}
C \tau^2 \int_M |v|^2 \tvol - C' \tau^2 \int_{\partial M} |v|^2 \hspace{2 pt} dS_{g} \leq \int_M | 2 \tau \partial_1 v |^2 \tvol. \notag
\end{align}
Note that
\begin{align*}
\int_M | \grad v |_{g}^2 \tvol - \int_{\partial M} \overline{v} \nmlder v \hspace{2 pt} dS_{g}
- \tau^2 \int_M |v|^2 \tvol &= \int_M \overline{v} (-\laplace - \tau^2 ) v \tvol \notag \\
&\leq \delta \int_M |v|^2 \tvol + \frac{1}{4 \delta} \int_M |(-\laplace - \tau^2)v|^2 \tvol \notag
\end{align*}
for arbitrary $\delta > 0$. Combining the above estimates and choosing $\delta > 0$ sufficiently small implies that
\begin{align}
C&\left ( \tau^2 \int_M |v|^2 \tvol + \int_M |\grad v |_{g}^2 \tvol \right ) - C' \tau^2 \int_{\partial M} |v|^2 \hspace{2 pt} dS_{g}
 - C'' \int_{\partial M} \bar v \nmlder v \hspace{2 pt} dS_{g} \notag \\
& \leq  \int_M |(-\laplace - \tau^2)v|^2 \tvol + \int_M | 2 \tau \partial_1 v |^2 \tvol. \label{secest}
\end{align}
for $|\tau| \geq \tau_0$.
Finally, a short calculation using local coordinates on $M_0$ shows that
\begin{align*}
4 \tau \Re ((\laplace + \tau^2)v \overline{\partial_1 v} ) = \diver \left [ 4 \tau \Re ( \partial_1 \overline v \grad v) - 2 \tau (\grad \varphi)
|\grad v |_{g}^2 + 2 \tau^3 (\grad \varphi) |v|^2 \right ],
\end{align*}
which by the divergence theorem implies that
\begin{align}
4 \tau \Re \int_M (\laplace + \tau^2) v \overline{\partial_1 v} \tvol &= \int_M \diver \left [ 4 \tau \Re ( \partial_1 \overline v \grad v) - 2 \tau (\grad \varphi)
|\grad v |_{g}^2 + 2 \tau^3 (\grad \varphi) |v|^2 \right ] \tvol \notag \\
&= \int_{\partial M} \left ( 4 \tau \Re (\partial_1 \bar v \nmlder v ) - 2 \tau (\nmlder \varphi) | \grad v |_{g}^2 + 2 \tau^3 (\nmlder \varphi)
|v|^2 \right ) \hspace{2 pt} dS_{g}. \label{lastest}
\end{align}
Combining \eqref{estone} with \eqref{secest} and \eqref{lastest} yields \eqref{carlest1}.
\end{proof}

For each value of the parameter $\tau$, let $q_\tau,A_\tau \in L^\infty(M)$.  An immediate corollary of the previous proposition is a Carleman estimate for the operator $-\laplace + \la A_\tau , \grad \ra_{g} + q_\tau$ as long as the norms $\| A_\tau \|_{L^\infty(M)}$ and $\| q_\tau \|_{L^\infty(M)}$ do not grow too
quickly in $\tau$.  The precise formulation is as follows.

\begin{cor}
For each value of the parameter $\tau \in \R$, let $q_\tau,A_\tau \in L^\infty(M)$ be such that
\begin{align*}
\| A_\tau \|_{L^\infty(M)}^2 + \tau^{-2} \| q_\tau \|_{L^\infty(M)}^2 = o(1), \quad \mbox{as } |\tau| \rightarrow \infty.
\end{align*}
Then there exist positive constants $\tau_0,C, C',C''$ with $\tau_0$ depending on the $o(1)$ term above such that for all $|\tau| \geq \tau_0$ and for all $v \in H^2(M)$ we have the estimate
\begin{align}
C&\left ( \tau^2 \int_M |v|^2 \tvol + \int_M |\grad v |_{g}^2 \tvol \right ) - C' \tau^2 \int_{\partial M} |v|^2 \hspace{2 pt} dS_{g} - C'' \int_{\partial M} \bar v \nmlder v \hspace{2 pt} dS_{g} \notag\\
&+ \int_{\partial M} \left ( 4 \tau \Re (\nmlder v \partial_1 \bar v - 2 \tau (\nmlder \varphi) | \grad v |_{g}^2 + 2 \tau^3 (\nmlder \varphi) |v|^2 \right ) \hspace{2 pt} dS_{g} \notag\\
&\leq \int_M \left |e^{-\tau \varphi } \left (-\laplace + \la A, \grad \ra_{g} + q \right ) (e^{\tau \varphi} v) \right |^2 \tvol. \label{carlest2}
\end{align}
\end{cor}

\section{CGO Solutions}

The main method for constructing CGO solutions to \eqref{condeq1} and \eqref{condeq2} is to first conjugate in the equation $\gamma^{1/2}$, i.e., $u$ solves
\eqref{condeq1} if and only if $v = \gamma^{-1/2}u$ satisfies
\begin{align}
(-\laplace + q)v = 0. \notag
\end{align}
where $q = \gamma^{-1/2}\Delta_g \left (\gamma^{1/2} \right )$.  If $(M,g)$ is an admissible manifold, then constructing CGO solutions to \eqref{condeq1} and \eqref{condeq2}
is therefore reduced to constructing CGO solutions to a zeroth order perturbation of the Laplacian $\laplace = \partial_1^2 + \Delta_{g_0}$.  However, the first step requires more smoothness that we assume on $\gamma$.  Following \cite{knu}, we proceed by introducing a smooth approximation
of $\gamma$ and conjugate the equation with the approximation.

Let $\gamma \in W^{3/2 + \eta, 2n} (M)$ be positive.  We first extend $\gamma$ to a function outside $M$ such that $\gamma - 1 \in W^{3/2+\eta, 2n}_0(\R \times M_0^{int})$.  Let $\phi = \log \gamma$ and $A = \grad \phi$.
We recall the following approximation result for functions $W^{3/2+\eta, 2n}_0(\R \times M_0^{int})$ from \cite{knu} (see Lemma 3.1).
\begin{lem}
Suppose $\gamma \in W^{3/2+\eta, 2n}(M)$ for some $\eta > 0$.  Then there exists a family of $C^\infty_0(\R \times M_0^{int})$ functions $\{ \phi_\tau \}_{\tau > 0}$ such that if $A_\tau = \grad \phi_\tau$ then as $\tau \rightarrow \infty$
\begin{align}
\| A - A_\tau \|_{L^\infty(M)} &= o(1), \label{AAtaulinfty}\\
\| \diver A_\tau \|_{L^\infty(M)} &= o(\tau), \label{diverlinfty}
\end{align}
and
\begin{align}
\| A - A_\tau \|_{L^2(M)} &= o(\tau^{-1/2-\eta}), \label{AAtaul2} \\
 \| \diver A_\tau \|_{L^2(M)} &= o(\tau), \label{diverl2}
\end{align}
Moreover,
\begin{align}
\| \gamma - e^{\phi_\tau} \|_{L^\infty(M)} &= O(\tau^{-1-\eta}), \label{pptaulinfty}\\
\| \grad(\gamma - e^{\phi_\tau}) \|_{L^\infty(M)} &= O(\tau^{-\eta}). \label{gradpptaulinfty}
\end{align}
\end{lem}

We remark that Lemma 3.1 from \cite{knu} is stated for $M_0^{int} = \R^n$ with the Euclidean gradient and divergence.
By a standard partition of unity argument, these estimates carry over to the current setting with the gradient and
divergence defined using $g$.

The CGO solutions to the equation \eqref{condeq1} we will construct are of the form
\begin{align}
u(x) = e^{-\phi_\tau/2} e^{- \tau x_1} \left ( v_\tau + \tilde r \right )
\end{align}
where $\psi, v_\tau \in C^\infty(M)$ will be chosen and $\tilde r$ is a remainder term that tends to zero in a suitable norm as $|\tau| \rightarrow \infty$ (along a sequence).
For the construction, fix a slightly larger simple manifold $\tilde M_0$ such that $M_0 \subseteq \tilde M_0^{int}$. We denote the determinant
of the metric $g_0$ by $|g_0|$.  The main result of the section is the following.

\begin{ppn}\label{CGO}
Suppose $\gamma \in W^{3/2+\eta,2n}(M)$, $\eta > 0$.  Let $\omega \in \tilde M_0 \setminus M_0$ be a fixed point, $\lambda \in \R$ be fixed, and let $b \in C^{\infty}(S^{n-2})$.  Write $x = (x_1,r,\theta)$
where $(r,\theta)$ are polar normal coordinates with center $\omega$.  For $|\tau|$ sufficiently large outside a countable set, there exists
a solution $u \in H^1(M)$ of $\diver(\gamma \grad u ) = 0$ in $M$ of the form
\begin{align*}
u = e^{-\phi_\tau/2} e^{-\tau x_1} \left ( e^{-i\tau r} |g_0|^{-1/4} e^{i \lambda(x_1 + ir)}b(\theta) + \tilde r \right )
\end{align*}
where $\tilde r$ satisfies
\begin{align}
\| \tilde r \|_{H^s(M)} = o(\tau^{-1/2 - \eta + s}), \quad 0 \leq s \leq 2. \label{remestimates}
\end{align}
\end{ppn}

For the proof, we will need the following conjugation relations which are straightforward calculations:
\begin{align*}
e^{\phi_\tau /2} ( -\laplace + \la A , \grad \ra_g ) e^{-\phi_\tau /2 } = - \laplace + \la (A - A_\tau), \grad \ra_g + q_{\tau}
\end{align*}
where
\begin{align*}
q_{\tau} = -\frac{1}{2}\diver A_\tau - \frac{1}{4} | A_\tau |^2_g - \frac{1}{2} \la A, A_\tau \ra_g,
\end{align*}
and
\begin{align*}
e^{-\tau x_1}(- \laplace + \la A - A_\tau, \grad \ra_{g} + q_{\tau}) e^{\tau x_1} = - \Delta_{g,\tau} + \la A - A_\tau, \grad \ra_{g} +  \tilde q_\tau,
\end{align*}
where
\begin{align*}
-\Delta_{g,\tau} &= - \laplace - \tau^2 - 2 \tau \partial_1, \\
\tilde q_\tau &=  q_{\tau} + \tau \la A - A_\tau, \grad \varphi \ra_{g}.
\end{align*}
Note that by \eqref{diverlinfty} and \eqref{diverl2} we have that
\begin{align}
\| q_\tau \|_{L^\infty(M)} = \| \tilde q_\tau \|_{L^\infty(M)} &= o(|\tau|), \label{potlinfty}\\
\| q_\tau \|_{L^2(M)} = \| \tilde q_\tau \|_{L^2(M)} &= o(|\tau|^{1/2-\eta}), \label{potl2}
\end{align}
as $|\tau| \rightarrow \infty$.

If we take as our ansatz $v_\tau = e^{-i \tau \psi} a$ for some $a \in C^\infty(M)$, then we have that $u$ solves \eqref{condeq1} if and only if
\begin{align}
(- \Delta_{g,\tau} + \la  A - A_\tau, \grad \ra_{g} + \tilde q_\tau ) (e^{i \tau \psi} \tilde r) = f, \label{remaindereq}
\end{align}
where $f = e^{i \tau \psi} (- \Delta_{g,\tau} + \la  A - A_\tau, \grad \ra_{g} + \tilde q_\tau) (e^{-i \tau \psi} a)$.  It was established in \cite{ksau} that $-\Delta_{g, \tau}$
has a bounded right inverse in appropriate spaces with norm decaying like $|\tau|^{-1}$ for $\tau$ outside a discrete set.  In particular, the following proposition holds.

\begin{ppn} For $|\tau| \geq 4$ outside a discrete set, there is a linear operator $G_{0,\tau} : L^2(M) \rightarrow H^2(M)$ such that
\begin{align*}
-\Delta_{g, \tau} G_{0,\tau} v = v \quad \mbox{for } v \in L^2(M).
\end{align*}
This operator satisfies
\begin{align*}
\|G_{0,\tau} f \|_{H^s(M)} \leq C |\tau|^{-1+s} \| f \|_{L^2(M)}, \quad 0 \leq s \leq 2.
\end{align*}
\end{ppn}

Proceeding perturbatively, we obtain an inverse for the operator $ - \Delta_{g,\tau} + \la A - A_\tau, \grad \ra_{g} + \tilde q_\tau$.

\begin{cor}There exists a $\tau_0 > 0$ such that for $|\tau| \geq \tau_0$ outside a discrete set, there is a linear operator $G_\tau : L^2(M) \rightarrow H^2(M)$ such that
\begin{align*}
\left (- \Delta_{g,\tau} + \la A - A_\tau, \grad \ra_{g} + \tilde q_\tau \right ) G_\tau v = v \quad \mbox{for } v \in L^2(M).
\end{align*}
This operator satisfies
\begin{align*}
\|G_\tau f \|_{H^s(M)} \leq C |\tau|^{-1+s} \| f \|_{L^2(M)}, \quad 0 \leq s \leq 2.
\end{align*}
\end{cor}

\begin{proof}
Given $f \in L^2(M)$, we wish to solve
\begin{align}
\left (- \Delta_{g,\tau} + \la A - A_\tau, \grad \ra_{g} + \tilde q_\tau \right ) v = f. \label{modeq}
\end{align}
We take as our ansatz
$$
v = G_{0,\tau} g.
$$
Then $v$ solves \eqref{modeq} if and only if $g$ solves the integral equation
\begin{align}
(I + K) g = f \label{integralequation}
\end{align}
where
$$
K = \left (\la A - A_\tau, \grad \ra_{g} + \tilde q_\tau \right ) G_{0,\tau}.
$$
Now we estimate
\begin{align*}
\| K \|_{L^2(M) \rightarrow L^2(M)} &\lesssim \| A - A_\tau \|_{L^\infty(M)} \| G_{0, \tau} \|_{L^2(M) \rightarrow H^1(M)} + \| \tilde q_\tau \|_{L^\infty(M)}
\| G_{0,\tau} \|_{L^2(M) \rightarrow L^2(M)} \\
&\lesssim \| A - A_\tau \|_{L^\infty(M)} + \frac{\| \tilde q_\tau \|_{L^\infty(M)}}{|\tau|} \\
&= o(1).
\end{align*}
Hence for $|\tau|$ sufficiently large and outside a discrete set we have that $\| K \|_{L^2(M) \rightarrow L^2(M)} < 1$. Thus, we may solve \eqref{integralequation}
for $g$ via a Neumann series and obtain $v$.  The norm estimates for $v$ follow from the corresponding norm estimates for $G_{0,\tau}$.

\end{proof}

The proof of Proposition \ref{CGO} now follows easily from the previous preparations.

\begin{proof}[Proof of Proposition \ref{CGO}.]
Let $\omega \in \tilde M_0 \setminus M_0$ be a fixed point, $\lambda \in \R$ be fixed, and let $b \in C^{\infty}(S^{n-2})$.  Write $x = (x_1,r,\theta)$
where $(r,\theta)$ are polar normal coordinates with center $\omega$.  If we choose
\begin{align*}
\psi(x_1,r,\theta) &= r, \\
a(x_1,r, \theta) &= |g_0|^{-1/4} e^{i \lambda(x_1 + ir)}b(\theta),
\end{align*}
Then it can be shown (see \cite{dfksu}) that
\begin{align*}
\|e^{i\tau \psi} (-\Delta_{g, \tau} ) (e^{-i \tau \psi} a)\|_{L^2(M)} = O(1).
\end{align*}
Moreover from \eqref{AAtaul2} and \eqref{potl2} we have that
\begin{align*}
\|(\la A-A_\tau, \grad \ra_{g} + \tilde q_\tau)(e^{-i \tau \psi} a) \|_{L^2(M)} &\lesssim |\tau| \| A - A_\tau \|_{L^2(M)} + \| \tilde q_\tau \|_{L^2(M)} \\
&= o(|\tau|^{1/2 - \eta}).
\end{align*}
Let $f = e^{i \tau \psi} (- \Delta_{g,\tau} + \la  A - A_\tau, \grad \ra_{g} + \tilde q_\tau) (e^{-i \tau \psi} a)$.  The previous
estimates imply that
$$
\| f \|_{L^2(M)} = o(|\tau|^{1/2-\eta})
$$
Recall that $u$ solves \eqref{condeq1} if and only if \eqref{remaindereq} holds.  Thus, setting
$$
\tilde r = e^{-i\tau \psi} G_\tau f
$$
yields the desired CGO solution with norm estimates.
\end{proof}

\section{Uniqueness for Partial Data}

We first state a result from \cite{dfks}.

\begin{lem}\label{raytr}
Let $(M_0,g_0)$ be an $(n-1)$--dimensional simple manifold, and let $F \in \mathcal{E}'(M_0^{int})$.  Consider the pairing
\begin{align}
\left \la F, e^{-\lambda r} b(\theta) |g_0|^{-1/2} \right \ra_{M^{int}_0}, \label{raytrident}
\end{align}
where $(r,\theta)$ are polar normal coordinates in $(M_0,g_0)$ centered at some $\omega \in \partial M_0$.
Here $\la \cdot, \cdot \ra_{M_0^{int}}$ denotes the dual pairing on $M_0^{int}$
between $F$ and $C^\infty$ functions.  If $|\lambda|$ is sufficiently small, and if these pairings vanish for all $\omega \in \partial M_0$ and all
$b \in C^\infty(S^{n-2})$, then $F = 0$.
\end{lem}

Lemma \ref{raytr} is closely connected to the injectivity of the attenuated ray transform on $M_0$ for small attenuations. For
this connection and the proof of Lemma \ref{raytr}, we refer the reader to Appendix A.

Finally, we have the following boundary integral identity from \cite{knu}.

\begin{lem}\label{bdyident}
Suppose $\gamma_j \in C^1(M)$, $j = 1,2$ and suppose $u_1, u_2 \in H^1(M)$ satisfy $\diver(\gamma_j \grad u_j) = 0$ in $M$.  Suppose further that
$\tilde u \in H^1(M)$ satisfies $\diver(\gamma_1 \grad \tilde u_1) = 0$ with $\tilde u_1 = u_2$.  Then
\begin{align*}
\int_M \left \la \gamma_1^{1/2} \grad \left ( \gamma_2^{1/2} \right ) - \gamma_1^{1/2} \grad \left ( \gamma_2^{1/2} \right ),
\grad(u_1 u_2) \right \ra_g \vol = \int_{\partial M} \gamma_1 \nmlder(\tilde u_1 - u_2 ) u_1 \hspace{2 pt} dS_g
\end{align*}
\end{lem}

The above result is stated in \cite{knu} for $M \subseteq \R^n$, but their argument is based on the divergence theorem and easily carries over to the current
setting.

For $\tau > 0$ sufficiently large, let
\begin{align*}
u_1 &= e^{-\phi_{1,\tau}/2} e^{-\tau x_1} \left ( e^{-i\tau r} |g_0|^{-1/4} e^{i \lambda(x_1 + ir)}b(\theta) + \tilde r_1 \right ), \\
u_2 &= e^{-\phi_{2,\tau}/2} e^{\tau x_1} \left ( e^{i\tau r} |g_0|^{-1/4} + \tilde r_2 \right ),
\end{align*}
be our CGO solutions to $\diver(\gamma_j \grad u_j) = 0$ in $M$, $j = 1,2$ which satisfy \eqref{remestimates}. In the notation from Section 3, we chose
\begin{align*}
v_{1,\tau} &= e^{-i\tau r} |g_0|^{-1/4} e^{i \lambda(x_1 + ir)}b(\theta), \\
v_{2,\tau} &= e^{i\tau r} |g_0|^{-1/4}.
\end{align*}
  Further, define $\tilde u_1$ as the
solution to $\diver(\gamma_1 \grad \tilde u_1) = 0$ in $M$ with $\tilde u_1 = u_2$ on $\partial M$.  Then Lemma \ref{bdyident} implies that
\begin{align}
\int_M \left \la \gamma_1^{1/2} \grad \left ( \gamma_2^{1/2} \right ) - \gamma_1^{1/2} \grad \left ( \gamma_2^{1/2} \right ),
\grad(u_1 u_2) \right \ra_g \vol = \int_{\partial M} \gamma_1 \nmlder(\tilde u_1 - u_2 ) u_1 \hspace{2 pt} dS_g \label{bdyform}
\end{align}

\begin{lem}\label{rhsofbdyident}
Suppose the assumptions of Theorem \ref{thm1} hold, and define $\tilde u_1, u_2,$ and $u_1$ as above.  Then
\begin{align}
\int_{\partial M} \gamma_1 \nmlder(\tilde u_1 - u_2 ) u_1 \hspace{2 pt} dS_g = o(1) \notag
\end{align}
as $\tau \rightarrow \infty$.
\end{lem}

\begin{proof}
The proof is similar to the proof from \cite{knu}.  The fact that $\gamma_1|_{\partial M} = \gamma_2|_{\partial M}$ and $\Lambda_{\gamma_1}|_{\partial M_{-,\epsilon}} = \Lambda_{\gamma_2}|_{\partial M_{-,\epsilon}}$
implies that $\nmlder(\tilde u_1 - u_2 ) = 0$ on $\partial M_{-,\epsilon}$.   By Cauchy-Schwarz
\begin{align}
\left | \int_{\partial M} \gamma_1 \nmlder(\tilde u_1 - u_2 ) u_1 \hspace{2 pt} dS_g \right |^2 &\lesssim \left ( \int_{\partial M_{+,\epsilon}} e^{-2\tau x_1}
| \nmlder(\tilde u_1 - u_2) |^2 \hspace{2 pt} dS_g \right) \left ( 1 + \| \tilde r_1 \|_{L^2(\partial M)}^2 \right ) \label{est1}
\end{align}
By the trace theorem and \eqref{remestimates}, we have for $0 < \delta < \eta$
$$
\| \tilde r_1 \|_{L^2(\partial M)} \lesssim \| \tilde r_1 \|_{H^{1/2+\delta}} \lesssim \tau^{\delta - \eta} = o(1).
$$
To estimate \eqref{est1} further we introduce the function
\begin{align*}
u = \left ( e^{\phi_{1,\tau}/2} \tilde u_1 - e^{\phi_{2,\tau}/2} u_2 \right ) = u_0 + \delta u,
\end{align*}
where
\begin{align*}
 u_0 &= e^{\phi_{1,\tau}/2}(\tilde u_1 - u_2 ), \\
\delta u &= \left ( e^{\phi_{1,\tau}/2} - e^{\phi_{2,\tau}/2} \right ) u_2.
\end{align*}
Since $\phi_{1,\tau}$ for large $\tau$ is uniformly bounded from below by a positive constant and $\tilde u_1 = u_2$ on $\partial M$, it follows that
\begin{align}
\int_{\partial M_{+,\epsilon}}e^{-2\tau x_1} | \nmlder (\tilde u_1 - u_2) |^2 \hspace{2 pt} dS_g \lesssim  \int_{\partial M_{+,\epsilon}} e^{-2 \tau x_1}
|\nmlder (\delta u) |^2 \hspace{2pt} dS_{g} + \int_{\partial M_{+,\epsilon}} e^{-2 \tau x_1}
|\nmlder (u) |^2 \hspace{2pt} dS_{g}. \label{est2}
\end{align}

\noindent \textbf{Claim 1:} We have as $\tau \rightarrow \infty$
\begin{align}
\int_{\partial M} e^{-2\tau x_1} |\delta u|^2 \hspace{2pt} dS_g &= O(\tau^{-2-2\eta}), \label{deltauest1} \\
\int_{\partial M} e^{-2\tau x_1} \left |\grad (\delta u) \right |_g^2 \hspace{2pt} dS_g &= O(\tau^{-2\eta}) \label{deltauest2}.
\end{align}
To prove the claim, observe $\gamma_1|_{\partial M} = \gamma_2|_{\partial M}$ and $\nmlder \gamma_1|_{\partial M} = \nmlder \gamma_2|_{\partial M}$, \eqref{pptaulinfty} and
\eqref{gradpptaulinfty} imply that
\begin{align*}
\|e^{\phi_{1,\tau}/2} - e^{\phi_{2,\tau}/2}\|_{L^\infty(\partial M)} &= O(\tau^{-1-\eta}), \\
\|\grad(e^{\phi_{1,\tau}/2} - e^{\phi_{2,\tau}/2}) \|_{L^\infty(\partial M)} &= O(\tau^{-\eta}).
\end{align*}
By the trace theorem, for $0 < \delta < \eta$ we have that
\begin{align*}
\|e^{-\tau x_1} u_2 \|_{L^2(\partial M)} &\lesssim  1 + \| \tilde r_2 \|_{H^{1/2+\delta}(M)} = O(1), \\
\|e^{-\tau x_1} \grad (u_2) \|_{L^2(\partial M)} &\lesssim \tau + \|\tilde r_2\|_{H^{3/2+\delta}(M)} = O(\tau).
\end{align*}
Hence
\begin{align*}
\int_{\partial M} e^{-2 \tau x_1} | \grad (\delta u) |_g^2 \hspace{2pt} dS_g &\lesssim
\|e^{\phi_{1,\tau}/2} - e^{\phi_{2,\tau}/2}\|^2_{L^\infty(\partial M)} \|e^{-\tau x_1} \grad (u_2) \|^2_{L^2(\partial M)} \\
&+ \|\grad(e^{\phi_{1,\tau}/2} - e^{\phi_{2,\tau}/2}) \|^2_{L^\infty(\partial M)} \|e^{-\tau x_1} u_2 \|^2_{L^2(\partial M)} \\
&= O(\tau^{-2\eta}).
\end{align*}
The bound for $\int_{\partial M} e^{-2\tau x_1} |\delta u|^2 \hspace{2pt} dS_g$ is proved similarly.  This proves \textbf{Claim 1}.

\noindent \textbf{Claim 2:} We have as $\tau \rightarrow \infty$
\begin{align}
\int_{\partial M_{+,\epsilon}} e^{-2 \tau x_1}
|\nmlder (u) |^2 \hspace{2pt} dS_{g} = o(1). \label{nmlderu}
\end{align}

We now apply our Carleman estimate \eqref{carlest2} with
\begin{align*}
A &= A_1 - A_{1,\tau}, \\
q &= q_{1,\tau}, \\
v &= e^{-\tau x_1} u.
\end{align*}
Note that \eqref{AAtaulinfty} and \eqref{potlinfty} imply the hypotheses for \eqref{carlest2} are satisfied.  Hence, we have that
\begin{align}
&\int_{\partial M_{+,\epsilon}} 4 \Re (\nmlder v \partial_1 \bar v) - 2  (\nmlder \varphi) | \grad v |_{g}^2 \hspace{2 pt} dS_{g} \lesssim
\tau \int_{\partial M} |v|^2 \hspace{2 pt} dS_{g} + \frac{1}{\tau} \left | \int_{\partial M} \bar v \nmlder v \hspace{2 pt} dS_{g} \right | \notag \\
&+ \tau^2 \int_{\partial M} \left | \nmlder \varphi \right | |v|^2 \hspace{2pt} dS_g +
\left | \int_{\partial M_{-,\epsilon}} 4 \Re (\nmlder v \partial_1 \bar v) - 2  (\nmlder \varphi) | \grad v |_{g}^2 \hspace{2 pt} dS_{g}\right | \notag \\
&+ \frac{1}{\tau} \int_M \left |e^{-\tau \varphi } \left (-\laplace + \la A_1-A_{1,\tau}, \grad \ra_{g} + q_{1,\tau} \right ) (e^{\tau \varphi} v) \right |^2 \tvol \label{carlest3}
\end{align}
We first estimate the left hand side of \eqref{carlest3}.  Young's inequality and $\tilde u_1 |_{\partial M} = u_2 |_{\partial M}$ imply
\begin{align}
\int_{\partial M_{+,\epsilon}} 4 \Re (\nmlder v \partial_1 \bar v) - 2  (\nmlder \varphi) | \grad v |_{g}^2 \hspace{2 pt} dS_{g} &\geq
\int_{\partial M_{+,\epsilon}} e^{-2\tau x_1} \left ( \nmlder \varphi \right ) \left | \nmlder u \right |^2 \hspace{2 pt} dS_{g} \notag \\
&- C \int_{\partial M_{+,\epsilon}} e^{-2\tau x_1} \left | \grad \delta u \right |_g^2 \hspace{2pt} dS_g \notag \\
&- C \tau^2 \int_{\partial M_{+,\epsilon}} e^{-2\tau x_1} |\delta u|^2 \hspace{2pt} dS_g \notag \\
&\geq  \epsilon \int_{\partial M_{+,\epsilon}} e^{-2\tau x_1} \left | \nmlder u \right |^2 \hspace{2 pt} dS_{g} + o(1) \label{carlest3lhs}
\end{align}
as $\tau \rightarrow \infty$.
We now estimate the right hand side of \eqref{carlest3}.  Since $\tilde u_1 |_{\partial M} = u_2 |_{\partial M}$, we have that
\begin{align}
\tau \int_{\partial M} |v|^2 \hspace{2pt} dS_g = \tau \int_{\partial M} e^{-2\tau x_1} |\delta u|\hspace{2pt} dS_g = O(\tau^{-1-2\eta}).\label{carlest3rhs1}
\end{align}
Similarly, we have that
\begin{align}
\tau^2 \int_{\partial M} \left | \nmlder \varphi \right | |v|^2 \hspace{2pt} dS_g = O(\tau^{-2 \eta}) \label{carlest3rhs2}
\end{align}
as $\tau \rightarrow \infty$.
Since $(\nmlder \tilde u_1 ) |_{\partial M_{-,\epsilon}} = (\nmlder u_2) |_{\partial M_{-,\epsilon}}$ as well, we have
\begin{align}
\left | \int_{\partial M_{-,\epsilon}} 4 \Re (\nmlder v \partial_1 \bar v) - 2  (\nmlder \varphi) | \grad v |_{g}^2 \hspace{2 pt} dS_{g} \right |
&\leq \tau^2 \int_{\partial M} e^{-2 \tau x_1} |\delta u|^2 \hspace{2pt} dS_g \notag \\
&+ \int_{\partial M} e^{-2 \tau x_1} |\grad(\delta u)|_g^2 \hspace{2pt} dS_g \notag \\
&=O(\tau^{-2 \eta}) \label{carlest3rhs3}
\end{align}
as $\tau \rightarrow \infty$.  Now by Young's inequality
\begin{align}
\frac{1}{\tau} \left | \int_{\partial M} \overline v \nmlder v \hspace{2pt} dS_g \right | &\leq
\left | \int_{\partial M} e^{-2 \tau x_1} |\delta u|^2 \nmlder \varphi \hspace{2pt} dS_g \right |
+ \frac{1}{\tau} \left | \int_{\partial M} e^{-2 \tau x_1} \overline{\delta u} \nmlder u \hspace{2pt} dS_g \right | \notag \\
&\leq O(\tau^{-2-2\eta}) + \frac{1}{\tau} \left ( \int_{\partial M} e^{-2 \tau x_1} |\delta u|^2 \hspace{2pt} dS_g \right )^{1/2}
\left ( \int_{\partial M} e^{-2 \tau x_1} |\grad (\delta u) |_g^2 \hspace{2pt} dS_g \right )^{1/2} \notag \\
&+ \frac{1}{2 \epsilon \tau^2} \int_{\partial M} e^{-2 \tau x_1} |\delta u|^2 \hspace{2pt} dS_g +
\frac{\epsilon}{2} \int_{\partial M_{+,\epsilon}} e^{-2 \tau x_1} |\nmlder u|^2 \hspace{2pt} dS_g \notag \\
&= o(1) + \frac{\epsilon}{2} \int_{\partial M_{+,\epsilon}} e^{-2 \tau x_1} |\nmlder u|^2 \hspace{2pt} dS_g \label{carlest3rhs4}
\end{align}
as $\tau \rightarrow \infty$.

We now estimate the final term appearing in the right hand side of \eqref{carlest3}.  Note that
\begin{align}
(-\laplace + \la A_1-A_{1,\tau},\grad \ra_g + q_1,\tau )\left (e^{\phi_{1,\tau}/2} \tilde u_1 \right ) &= 0, \notag \\
(\left (-\laplace + \la A_1 - A_{1,\tau},\grad \ra_g + q_1,\tau \right)\left (e^{\phi_{2,\tau}/2} u_2 \right ) &=
\left ( \la (A_1-A_{1,\tau}) - (A_2-A_{2,\tau}) \right ), \grad \ra_g \notag \\
&+ q_{1,\tau} - q_{2,\tau} ) \left ( e^{\tau x_1}(v_{2,\tau} + \tilde r_2) \right ) \label{carlest3rhs4a}
\end{align}
As $\tau \rightarrow \infty$
\begin{align}
\frac{1}{\tau} \int_M \left | \left  (q_{1,\tau} - q_{2,\tau} \right ) \left ( e^{\tau x_1}(v_{2,\tau} + \tilde r_2) \right ) \right |^2 \hspace{2pt} dV_g \lesssim
\frac{1}{\tau} \cdot o(\tau) \cdot \left ( 1 + o(\tau^{-1-2\eta}) \right ) = o(1), \label{carlest3rhs4b}
\end{align}
and for $j = 1,2$
\begin{align}
\frac{1}{\tau} \int_{M} \left | \la A_j - A_{j,\tau}, \grad \left ( e^{\tau x_1}(v_{2,\tau} + \tilde r_2) \right ) \ra_g \right |^2 &\vol \lesssim \tau \left \| A_j - A_{j,\tau} \right \|_{L^2(M)}^2  +  \tau \left \| A_j - A_{j,\tau} \right \|_{L^\infty(M)}^2 \left \| \tilde r_2 \right \|_{L^2(M)}^2 \notag \\
&+ \tau \left \| A_j - A_{j,\tau} \right \|_{L^2(M)}^2 + \frac{1}{\tau} \left \| A_j - A_{j,\tau} \right \|_{L^\infty(M)}^2 \left \| \tilde r_2 \right \|_{H^1(M)}^2 \notag \\
&= o(\tau^{-2\eta}) \label{carlest3rhs4c}
\end{align}
Then \eqref{carlest3rhs4a}--\eqref{carlest3rhs4c} yield
\begin{align}
\frac{1}{\tau} \int_M &\left |e^{-\tau \varphi } \left (-\laplace + \la A_1-A_{1,\tau}, \grad \ra_{g} + q_{1,\tau} \right )
(e^{\tau \varphi} v) \right |^2 \tvol = o(1) \label{carlest3rhs5}
\end{align}
as $\tau \rightarrow \infty$.
Combining \eqref{carlest3}--\eqref{carlest3rhs4}, and \eqref{carlest3rhs5} give \textbf{Claim 2}.

Finally, \eqref{est1},\eqref{est2}, \textbf{Claim 1}, and \textbf{Claim 2} yield Lemma \ref{rhsofbdyident}.
\end{proof}
We now prove Theorem \ref{thm1}.

\begin{proof}[Proof of Theorem \ref{thm1}.]
We now go back to \eqref{bdyident}.  It is a well-known fact that $u_1,u_2 \in H^{1/2}(M)$ implies $u_1u_2 \in W^{1/2,2n/(2n-1)}(M)$ with
\begin{align}
\| u_1 u_2 \|_{ W^{1/2,2n/(2n-1)}(M)} \leq C \| u_1 \|_{H^{1/2}(M)} \| u_2 \|_{H^{1/2}(M)}. \notag
\end{align}
It then follows from \eqref{pptaulinfty} and \eqref{remestimates} that as $\tau \rightarrow \infty$
\begin{align*}
u_1 u_2 &=  e^{-\phi_{1,\tau}/2} e^{-\phi_{2,\tau}/2} \left ( v_{1,\tau} v_{2,\tau} + v_{1,\tau} \tilde r_2
+ v_{2,\tau} \tilde r_1 + \tilde r_1 \tilde r_2 \right ) \\
&=  \gamma_1^{1/2} \gamma_2^{1/2} |g|^{-1/2} e^{i \lambda(x_1 + ir)}b(\theta) + o(\tau^{-\eta})
\end{align*}
in $W^{1/2,2n/(2n-1)}(M)$ and therefore
\begin{align*}
\grad(u_1 u_2) = \grad \left ( \gamma_1^{1/2} \gamma_2^{1/2} |g|^{-1/2} e^{i \lambda(x_1 + ir)}b(\theta) \right ) + o(\tau^{-\eta})
\end{align*}
in $W^{-1/2,2n/(2n-1)}(M)$.  Since $\gamma_1^{1/2} \grad \left ( \gamma_2^{1/2} \right ) - \gamma_1^{1/2} \grad \left ( \gamma_2^{1/2} \right)
\in  W_0^{1/2,2n/(2n-1)}(M)$, it follows from Lemmas \ref{bdyident} and \ref{rhsofbdyident} by taking $\tau \rightarrow \infty$
that
\begin{align*}
0 &= \int_M \left \la \gamma_1^{1/2} \grad \left ( \gamma_2^{1/2} \right ) - \gamma_1^{1/2} \grad \left ( \gamma_2^{1/2} \right ),
\grad(c \gamma_1^{1/2} \gamma_2^{1/2} |g_0|^{-1/2} e^{i \lambda(x_1 + ir)}b(\theta)) \right \ra_g \vol \\
&= \left \la F , e^{i \lambda(x_1 + ir)}b(\theta) |g_0|^{-1/2} \right \ra_{\R \times M_0^{int}}
\end{align*}
for all $\lambda \in \R$  and $b \in C^\infty(S^{n-2})$ where
\begin{align*}
F =  \diver \left (\gamma_1^{1/2} \grad \left ( \gamma_2^{1/2} \right ) - \gamma_1^{1/2} \grad \left ( \gamma_2^{1/2} \right )
\right ) \in \mathcal E'(\R \times M_0^{int}),
\end{align*}
and $\la \cdot, \cdot \ra_{\R \times M_0^{int}}$ denotes the dual pairing on $\R \times M_0^{int}$.  Written slightly differently, we have that
\begin{align*}
\left \la F_\lambda, e^{- \lambda r}b(\theta) |g_0|^{-1/2}   \right \ra_{M_0^{int}} = 0
\end{align*}
for all $\lambda \in \R$  and $b \in C^\infty(S^{n-2})$ where $F_\lambda \in \mathcal E'(M_0^{int})$ is the Fourier transform of $F$
in the $x_1$--direction, i.e.
\begin{align*}
\left \la F_\lambda, v \right \ra_{M_0^{int}} := \left \la F, e^{i\lambda x_1} v \right \ra_{\R \times M_0^{int}}, \quad v \in C^\infty(M_0^{int}).
\end{align*}
By Lemma \ref{raytr}, we have that $F_\lambda = 0$ for $\lambda$ sufficiently
small.  By the Paley-Weiner theorem, this implies that $F = 0$.  Therefore
\begin{align*}
\frac{1}{2} \laplace \left ( \log \gamma_1 - \log \gamma_2 ) \right ) + \frac{1}{4} \left \la \grad \left ( \log \gamma_1 + \log \gamma_2
\right ), \grad \left ( \log \gamma_1 - \log \gamma_2 \right ) \right \ra_g &= 0 \quad \mbox{in } M, \\
\log \gamma_1 - \log \gamma_2 &= 0 \quad \mbox{on } \partial M.
\end{align*}
Uniqueness for the boundary value problem implies that $\log \gamma_1 = \log \gamma_2$.
\end{proof}

\begin{proof}[Proof of Theorem \ref{thm2}.] Without loss of generality, we may take $x_0 = 0$, and we may assume that $\Omega \subseteq
\{ x \in \R^n : x_n > 0 \}$.
Then $\varphi(x) = \log |x|$.  In the notation of Theorem \ref{thm1}, we set $M = \Omega$ and $g$ equal to the induced Euclidean metric on $\Omega$. We make a change of coordinates
\begin{align}
y_1 &= \log |x|, \notag \\
y' &= \frac{x}{|x|}. \notag
\end{align}
The coordinates $y'$ parametrize the manifold $S^{n-1}$, and the Euclidean metric $g$ on $\Omega$ is given by
$$
g = c \left ( e \oplus g_0 \right )
$$
where $c(y_1,y') = \exp \left (2 y_1 \right )$, $e$ is the Euclidean metric on $\R$, and $g_0 = g_{S^{n-1}}$ is the standard round metric on $S^{n-1}$. Therefore
\begin{align}
(\Omega, g) = (M,g) \subseteq (\R \times M_0^{int}, g) \notag
\end{align}
where $(M_0,g_0) \subseteq (\{ \theta \in S^{n-1} : \theta_n > 0 \},g_0)$ is a closed \lq cap' which is a simple manifold. Moreover, $\varphi(y_1,y') = y_1$ whence in the notation
of Theorem \ref{thm1}
\begin{align}
\partial \Omega_{\pm,\epsilon} = \partial M_{\pm,\epsilon}. \notag
\end{align}
Theorem \ref{thm2} follows immediately from Theorem \ref{thm1}.
\end{proof}

\appendix

\section{}

In this appendix we prove Lemma \ref{raytrident}.  The proof is very similar to the proof of Lemma 5.1 from \cite{dfks}.  We first
introduce some notation and recall some facts.

Consider the unit sphere bundle
$$
SM_0 = \bigcup_{x \in M_0} S_x, \quad S_x = \{(x,\xi) \in T_xM_0 : |\xi| = 1 \}.
$$
This manifold has boundary $\partial (SM_0) = \{ (x,\xi) \in SM_0 : x \in \partial M_0 \}$ which is the union of the two sets
\begin{align*}
\partial_+ (SM_0) &= \{ (x, \xi) \in SM_0 : \la \xi, \nu \ra \leq 0 \} \quad \mbox{(inward pointing vectors)}, \\
\partial_- (SM_0) &= \{ (x, \xi) \in SM_0 : \la \xi, \nu \ra \geq 0 \} \quad \mbox{(outward pointing vectors)},
\end{align*}
where $\nu$ is the outer unit normal vector.

Denote by $t \rightarrow \gamma(t,x,\xi)$ the unit speed geodesic starting at $x$ in the direction $\xi$, and let
$\tau(x,\xi)$ denote the time when the geodesic exits $M_0$.  The exit time $\tau(x,\xi)$ is finite for all
$(x,\xi) \in SM_0$ since
$M_0$ is simple.  We also denote the geodesic flow by
$\varphi_t(x,\xi) = (\gamma(t,x,\xi), \dot \gamma(t,x,\xi))$.

The geodesic ray transform with constant attenuation $-\lambda$, acts on $C^\infty$ functions on $M_0$
by
$$
T_\lambda f(x, \xi) = \int_0^{\tau(x,\xi)} f(\gamma(t,x,\xi)) e^{-\lambda t} dt, \quad (x,\xi) \in \partial_+(SM_0).
$$

Suppose the hypotheses of Lemma \ref{raytr} are satisfied, and $F = f \in C^\infty_0(M_0^{int})$.  Suppose $(r,\theta)$ are polar normal coordinates in $(M_0,g_0)$ centered at
some $\omega \in \partial M_0$.  Since $M_0$ is simple, $(r,\theta)$ are global coordinates on $M_0$. Then the hypotheses of
Lemma \eqref{raytrident} read
\begin{align*}
 \int_{S^{n-2}} \int_0^{\tau(\omega,\theta)}f(r,\theta) e^{-\lambda r} b(\theta) \hspace{2pt} dr \hspace{2pt} d\theta = 0
\end{align*}
for all $b \in C^{\infty}(S^{n-2})$. By varying
$b$, we obtain
\begin{align*}
T_\lambda f(\omega, \theta_0) = \int_0^{\tau(\omega,\theta_0)}f(r,\theta_0) e^{-\lambda r} \hspace{2pt} dr = 0
\end{align*}
for any $\omega \in \partial M_0$ and any $\theta_0 \in S^{n-2}$. Hence the attenuated ray transform $T_\lambda f$ of $f$ is identically zero for small $\lambda$.  By the
following injectivity result (see \cite{dfksu}, Theorem 7.1), it follows that $f = 0$.

\begin{ppn}
Let $(M_0,g_0)$ be a simple manifold.  There exists $\epsilon > 0$ such that if $\lambda$ is a real number with
$|\lambda| < \epsilon$, and if $f \in C^\infty(M_0)$, then the condition $T_\lambda f(x,\xi) = 0$ for all $(x,\xi) \in
\partial_+(SM_0)$ implies that $f = 0$.
\end{ppn}

We can reduce the case when $F$ is a distribution that is compactly supported in $M_0^{int}$
to the previous case $F \in C^\infty_0(M_0^{int})$ by using duality and the ellipticity
of the operator $T_\lambda^* T_\lambda$.  We will need a few facts about $T_\lambda$ and $T^*_\lambda$.

We write
$$
h_{\varphi}(x,\xi) = h(\varphi_{-\tau(x, -\xi)} (x, \xi), \quad (x,\xi) \in SM_0
$$
for $h \in C^\infty(\partial_+(SM_0))$, and
$$
(h, \tilde h)_{L^2_\mu(\partial_+(SM_0))} = \int_{\partial_+(SM_0)} h \tilde h \mu \hspace{2pt} d(\partial(SM_0)),
$$
where $\mu(x,\xi) = -\la \xi, \nu(x) \ra$ and $dN$ is the Riemannian volume form on a manifold $N$.  From Section 5
of \cite{dfks}, we have the following facts.

\begin{lem}[\emph{Santal\'o Formula}]\label{santform}
If $F : SM_0 \rightarrow \R$ is continuous then
\begin{align*}
\int_{SM_0} F \hspace{2pt} &d(SM_0) \\
&= \int_{\partial_+(SM_0)} \int_0^{\tau(x,\xi)} F(\gamma(t,x,\xi)) \mu(x,\xi) \hspace{2pt} dt \hspace{2pt}
d(\partial(SM_0))(x,\xi).
\end{align*}
\end{lem}

\begin{lem}\label{rayadj}
If $f \in C^\infty(M_0)$ and $h \in C^\infty((\partial_+(SM_0))^{int})$ then
$$
(T_\lambda f, h)_{L^2_\mu(\partial_+(SM_0))} = (f, T^*_\lambda h)_{L^2(M_0)}
$$
where
$$
T^*_\lambda h(x) = \int_{S_x} e^{-\lambda \tau(x, -\xi)} h_{\varphi}(x,\xi) \vspace{2pt} dS_x(\xi), \quad
x \in M_0.
$$
\end{lem}

\begin{lem}
$T^*_\lambda T_\lambda$ is a self--adjoint elliptic pseudodifferential operator of order $-1$ in $M_0^{int}$.
\end{lem}

We now turn to the proof of Lemma \ref{raytr}.

\begin{proof}[Proof of Lemma \ref{raytr}]
Since $F \in \mathcal{E}'(M_0^{int})$, we can choose a sequence $f_n \in C^\infty_0(M_0^{int})$ such that
\begin{align}\label{appa1}
\la F, g \ra_{M^{int}_0} = \lim_{n \rightarrow \infty} (f_n , g)_{L^2(M_0)}, \quad \forall g \in C^\infty(M_0)
\end{align}
and so that there exists a constant $C_0 > 0$ and integer $N \geq 1$ such that for all $g \in C^\infty(M_0)$,
\begin{align}\label{appa2}
|(f_n , g)_{L^2(M_0)}| \leq C_0 \| g \|_{C^N(M_0)}
\end{align}
uniformly in $n$.

Let $b$ from \eqref{raytrident} depend also on $\omega$.  By changing notation, we see that
\eqref{raytrident} and \eqref{appa1} imply
$$
\lim_{n \rightarrow \infty} \int_{S_x} \int_0^{\tau(x,\xi)} e^{-\lambda t} f_n(\gamma(t,x,\xi)) b(x,\xi) \hspace{2pt}
dt \hspace{2pt} dS_x(\xi) = 0
$$
for all $x \in \partial M_0$ and $b \in C^\infty_0((\partial_+(SM_0))^{int})$.  We choose $b(x,\xi) = h(x,\xi)
\mu(x,\xi)$ for $h \in C^\infty_0((\partial_+(SM_0))^{int})$, integrate the last identity over $\partial M_0$, and use
\eqref{appa2} and the dominated convergence theorem to obtain
$$
\lim_{n \rightarrow \infty} \int_{\partial_+(SM_0)} \int_0^{\tau(x,\xi)} e^{-\lambda t}
f_n(\gamma(t,x,\xi)) h(x,\xi) \mu \hspace{2pt} dt \hspace{2pt} d(\partial(SM_0)) = 0.
$$
By Lemma \ref{santform} and \ref{rayadj} we have for each $n$
\begin{align*}
\int_{\partial_+(SM_0)}& \int_0^{\tau(x,\xi)} e^{-\lambda t}
f_n(\gamma(t,x,\xi)) h(x,\xi) \mu \hspace{2pt} dt \hspace{2pt} d(\partial(SM_0)) \\ &=
\int_{\partial_+(SM_0)} \int_0^{\tau(x,\xi)} e^{-\lambda t}
f_n(\varphi_t(x,\xi)) h_{\varphi}(\varphi_t(x,\xi)) \mu \hspace{2pt} dt \hspace{2pt} d(\partial(SM_0)) \\
&= \int_{SM_0} e^{-\lambda \tau(x,-\xi)} f_n(x) h_{\varphi}(x,\xi) \hspace{2pt} d(SM_0) \\
&= \int_{M_0} f_n(x) \left ( \int_{S_x} e^{-\lambda \tau(x,-\xi)} f_n(x) h_{\varphi}(x,\xi) \hspace{2pt} dS_x(\xi) \right )
\hspace{2pt} dV(x) \\
&= (f_n, T^*_\lambda h)_{L^2(M_0)}.
\end{align*}
We let $n \rightarrow \infty$ and use \eqref{appa1} to deduce
$$
\la F, T^*_\lambda h \ra_{M_0^{int}} = 0
$$
for all $h \in C^{\infty}_0((\partial_+(SM_0))^{int})$.

We now choose $h = T_\lambda \psi$ for $\psi \in C^\infty_0(M^{int})$
to obtain
$$
\la F, T^*_\lambda T_\lambda \psi \ra_{M_0^{int}} = 0.
$$
Since $T^*_\lambda T_\lambda$ is self--adjoint, the previous line implies that $T^*_\lambda T_\lambda F = 0$.  By ellipticity,
there exists a pseudodifferential operator of order $1$ in $M_0^{int}$ denoted by $B$ and a smoothing operator $R :
\mathcal{E}'(M_0^{int}) \rightarrow C^\infty(M_0^{int})$ so that $B T^*_\lambda T_\lambda = I + R$.  Thus, $T^*_\lambda T_\lambda F = 0$
implies that $F = -R F \in C^{\infty}(M_0^{int})$.  We can now use the argument for smooth $F$ from before to conclude the
proof that $F = 0$.
\end{proof}

\end{document}